\documentclass{amsart}
\usepackage{amssymb}
\usepackage{amsmath}
\usepackage{amsfonts}

\newtheorem{theorem}{Theorem}
\theoremstyle{plain}

\newtheorem{corollary}{Corollary}

\newtheorem{proposition}{Proposition}
\newtheorem{remark}{Remark}

\numberwithin{equation}{section}
\input{tcilatex}

\begin{document}
\title{ON THE\ EXISTENCE OF PROPER NEARLY KENMOTSU MANIFOLDS}
\author{I. K\"{u}peli Erken}
\address{ Uludag University, Faculty of Art and Science, Department of
Mathematics, Gorukle 16059, Bursa-TURKEY}
\email{iremkupeli@uludag.edu.tr, piotrdacko@yahoo.com, cengiz@uludag.edu.tr}
\author{Piotr Dacko}
\author{C. Murathan }

\begin{abstract}
This is an expository paper, which provides a first approach to nearly
Kenmotsu manifolds. The purpose of this paper is to focus on nearly Kenmotsu
manifolds and get some new results from it. We prove that for a nearly
Kenmotsu manifold is locally isometric to warped product of real line and
nearly K\"{a}hler manifold. Finally, we prove that there exist no nearly
Kenmotsu hypersurface $M^{2n+1}$ of nearly K\"{a}hler manifold $N^{2n+2}$.
It is shown that a normal nearly Kenmotsu manifold is Kenmotsu manifold.
\end{abstract}

\maketitle

% no mail address for me

\section{I\textbf{ntroduction}}

Nearly Kaehler manifolds were defined by Gray \cite{gray1}. He carry on
study of Nearly Kaehler manifolds \cite{gray1.5}, \cite{gray2}. Nearly
Sasakian manifolds were introduced by Blair, Showers and Yano in \cite%
{Blairnearly}. Afterwards, Olszak studied nearly Sasakian manifolds in \cite%
{olszaktensor}. He gave properties of $5$-dimensional nearly Sasakian
non-Sasakian manifolds. Parallel to the study of \cite{olszaktensor}, Endo
studied nearly cosymplectic manifolds \cite{endo}. Recently, Cappelletti
Montano and Dileo study Nearly Sasakian Geometry \cite{mino}. While much of
the similarity between nearly Sasakian manifolds and nearly cosymplectic
manifolds are emphasized, and properties of these manifolds are
investigated, nearly Kenmotsu manifolds are ignored. The notion of nearly
Kenmotsu manifold was introduced in \cite{shukla}. In the present paper, we
want to fill this gap in the study of nearly Kenmotsu manifolds. In
literature we did not fall in with proper nearly Kenmotsu manifold examples.
\ So one can ask the following question. Do there exist proper nearly
Kenmotsu manifolds? In this paper we give a positive answer to the question
for dimension $>5$. In this study we gave certain properties of such
manifolds. Our work is structured as follows: In Section $2$, we report some
basic information about nearly Kenmotsu manifolds. In the next section, we
give some curvature identies about nearly Kenmotsu manifolds and we prove
that for a nearly Kenmotsu manifold, $H=0$ and the distribution $D$ is
completely integrable. In the last section, we show that a normal nearly
Kenmotsu manifold is Kenmotsu manifold and there exist no nearly Kenmotsu
hypersurface $M^{2n+1}$ of nearly Kaehler manifold $N^{2n+2}$.

\section{Preliminaries}

\label{n.k.m.} In this paper all objects are to be considered as $C^\infty$%
-class, manifolds are assumed to be connected. We accept the following
convention that $X,Y, Z,W \ldots $, will denote vector fields, if it is not
otherwise stated.

Let $M$ be a $(2n+1)$-dimensional differentiable manifold and $\phi $ is a $%
(1,1)$ tensor field, $\xi $ is a vector field, $\eta $ is a one-form, $g$
Riemannina metric on $M$. Then $(\phi ,\xi ,\eta ,g)$ is called an almost
contact metric structure on $M$, if%
\begin{equation*}
\phi ^{2}=-Id+\eta \otimes \xi ,\quad \eta (\xi )=1,\quad g(\phi X,\phi
Y)=g(X,Y)-\eta (X)\eta (Y).
\end{equation*}%
and $M$ is said to be an almost contact metric manifold if it is endowed
with an almost contact metric structure \cite{Blair}, \cite{yano}. For such
manifold 
\begin{gather}
\eta (X)=g(X,\xi ),\quad \phi (\xi )=0,\quad \eta \circ \phi =0, \\
g(X,\phi Y)+g(Y,\phi X)=0,  \label{2}
\end{gather}%
tensor field $\Phi (X,Y)=g(X,\phi Y)$, is customary called fundamental form.
In this paper we will refer to $\xi $, as Reeb vector field and $\eta $, as
Reeb form. By $[\phi ,\phi ]$ we denote Nijenhuis torsion tensor of $\phi $,
by definition 
\begin{equation}
\lbrack \phi ,\phi ](X,Y)=\phi ^{2}[X,Y]+[\phi X,\phi Y]-\phi \lbrack X,\phi
,Y]-\phi \lbrack X,\phi Y],
\end{equation}%
where $[X,Y]$ denotes the Lie bracket of vector fields.

An almost contact metric manifold $(M,\phi ,\xi ,\eta )$ is called nearly
Kenmotsu manifold \cite{shukla}, if%
\begin{equation}
(\nabla _{X}\phi )Y+(\nabla _{Y}\phi )X=-\eta (Y)\phi X-\eta (X)\phi Y,
\label{4}
\end{equation}%
where $\nabla $ is the Levi-Civita connection of $g$. Moreover, if $M$
satisfies%
\begin{equation}
(\nabla _{X}\phi )Y=g(\phi X,Y)\xi -\eta (Y)\phi X,  \label{5}
\end{equation}%
then it is called Kenmotsu manifold \cite{kenmotsu}. Every Kenmotsu manifold
is a nearly Kenmotsu manifold but the converse is not true, which in fact
will be proved in this paper. If $M$ is nearly Kenmotsu but non Kenmotsu we
will call manifold is proper nearly Kenmotsu manifold.

Let $M$ be nearly Kenmotsu manifold. We define $(1,1)$-tensor field $H$, by $%
d\eta(X,Y)=g(HX,Y)$. Later on we will show that $H=0$.

\begin{proposition}
\label{killing}For a nearly Kenmotsu manifold we have 
\begin{eqnarray}
&&g(\nabla _{X}\xi ,Y)+g(X,\nabla _{Y}\xi )=2g(\phi X,\phi Y),\quad \nabla
_{X}\xi =-\phi ^{2}X+HX,  \label{8} \\
&&\nabla _{\xi }\phi =\phi H,\quad \phi H+H\phi =0,\text{ }H\xi =0,\quad
\nabla _{\xi }\xi =0.
\end{eqnarray}
\end{proposition}

\begin{proof}
By (\ref{4}), $(\nabla _{\xi }\phi )\xi =\phi \nabla _{\xi }\xi =0$, hence $%
\nabla _{\xi }\xi =0$, and $\nabla _{\xi }\eta =0$. Now, $g(\phi X,\phi
Y)=g(X,Y)-\eta (X)\eta (Y)$, yields%
\begin{eqnarray*}
&&0=g((\nabla _{\xi }\phi )X,\phi Y)+g((\nabla _{\xi }\phi )Y,\phi
X)=-g((\nabla _{X}\phi )\xi ,\phi Y)-g((\nabla _{Y}\phi )\xi ,\phi X) \\
{} &&-2g(\phi X,\phi Y)=g(\nabla _{X}\xi ,Y)+g(\nabla _{Y}\xi ,X)-2g(\phi
X,\phi Y).
\end{eqnarray*}%
With help of definition of $H$, $\nabla _{X}\xi =-\phi ^{2}X+HX$. By $\phi
\xi =0$, and $\eta (\phi X)=0$ 
\begin{eqnarray}
&&0=(\nabla _{X}\phi )\xi +\phi \nabla _{X}\xi =-(\nabla _{\xi }\phi )X+\phi
HX, \\
&&0=\eta ((\nabla _{X}\phi )Y)+\eta ((\nabla _{Y}\phi )X)=-g((\nabla
_{X}\phi )\xi ,Y)-g((\nabla _{Y}\phi )\xi ,X) \\
&=&g((\nabla _{\xi }\phi )X,Y)+g((\nabla _{\xi }\phi )Y,X)=g(\phi
HX,Y)+g(\phi HY,X)  \notag \\
&=&g((\phi H+H\phi )X,Y).  \notag
\end{eqnarray}
\end{proof}

% The Ricci identity can be written by 
% \begin{equation}
% g((\nabla _{X,Y}^{2}\phi )W,Z)-g((\nabla _{Y,X}^{2}\phi )W,Z)=g(R(X,Y)
% W,\phi Z)-g(R(X,Y)Z,\phi W),  \label{NABLASQUARE2}
% \end{equation}%
% where $R$ is the Riemann  curvature opeerator on $M.~$

\begin{proposition}
\label{2-form} The fundamental form satisfies 
\begin{gather}  \label{2-FORM1}
3d\Phi (X,Y,Z)=-3g((\nabla _{X}\phi )Y,Z)-\eta (Y)g(\phi X,Z)+ \eta
(Z)g(\phi X,Y) \\
{}-{}2\eta (X)g(\phi Y,Z).  \notag \\
d\Phi(X,Y,Z) -\eta(Z)(\nabla_\xi\Phi)(X,Y)= \dfrac{1}{4}g([\phi,\phi](X,Y),%
\phi Z) +2(\eta\wedge \Phi)(X,Y,Z).
\end{gather}
\end{proposition}

\begin{proof}
From identities 
\begin{gather}
3d\Phi(X,Y,Z)= (\nabla_X\Phi)(Y,Z)+(\nabla_Y\Phi)(Z,X)+(\nabla_Z\Phi)(X,Y),
\\
{} [\phi,\phi ](X,Y) = -\phi
(\nabla_X\phi)Y+\phi(\nabla_Y\phi)X+(\nabla_{\phi X}\phi)Y -(\nabla_{\phi
Y}\phi)X.
\end{gather}
we obtain 
\begin{multline}
3d\Phi (X,Y,Z) =-g((\nabla _{X}\phi )Y,Z)-g((\nabla_Z\phi)X,Y)+
g((\nabla_Y\phi)X,Z) \\
=-3g((\nabla _{X}\phi )Y,Z) {} -2\eta (X)g(\phi Y,Z)+\eta (Y)g(\phi
Z,X)-\eta (Z)g(\phi Y,X),
\end{multline}
\begin{equation}
\dfrac{1}{2}[\phi,\phi](X,Y) =
-\phi(\nabla_X\phi)Y+\phi(\nabla_Y\phi)X+\eta(Y)X-\eta(X)Y.
\end{equation}%
Hence 
\begin{multline*}
6d\Phi(X,Y,Z)= -3g((\nabla_X\phi)Y-(\nabla_Y\phi)X,Z) +\eta(Y)g(\phi X,Z)
{}-\eta(X)g(\phi Y,Z) + 2\eta(Z)g(\phi X,Y) \\
= \dfrac{3}{2}g([\phi,\phi](X,Y),\phi Z) +4\eta(X)g(Y,\phi Z) {}+{}
4\eta(Y)g(Z,\phi X)+4\eta(Z)g(X,\phi Y)+ \\
{} 6\eta(Z)(\nabla_\xi\Phi)(X,Y) = \dfrac{3}{2}g([\phi,\phi],\phi Z) + 12
(\eta\wedge\Phi)(X,Y,Z)+6\eta(Z)(\nabla_\xi\Phi)(X,Y).
\end{multline*}
\end{proof}

\section{Structure of nearly Kenmotsu Manifolds}

In this section, we will proof curvature relations for nearly Kenmotsu
manifold. Let $R$ be Riemann curvature operator 
\begin{equation*}
R(X,Y)Z = (\nabla^2_{X,Y}Z)-(\nabla^2_{Y,X}Z) =
[\nabla_X,\nabla_Y]Z-\nabla_{[X,Y]}Z,
\end{equation*}
by the same letter we denote corresponding $(0,4)$-tensor 
\begin{equation*}
R(X,Y,Z,W)=g(R(X,Y)Z,W).
\end{equation*}

\begin{theorem}
Let $(M,\phi ,\xi ,\eta ,g)$ be nearly Kenmotsu manifold. We have following
curvature relations 
\begin{gather}
R(\phi X,Y,Z,W)+R(X,\phi Y,Z,W)+R(X,Y,\phi Z,W)+R(X,Y,Z,\phi W)=0,
\label{FEN1} \\
R(\phi X,\phi Y,Z,W)=R(X,Y,\phi Z,\phi W),  \label{RRR} \\
R(\phi X,\phi Y,\phi Z,\phi W)=R(X,Y,Z,W)-\eta (X)R(\xi ,Y,Z,W)+\eta
(Y)R(\xi ,X,Z,W).  \label{R5}
\end{gather}
\end{theorem}

\begin{proof}
Let $T$ be $(1,3)$-tensor defined by (cf. (\ref{4})) 
\begin{equation}
(\nabla _{X,Y}^{2}\phi )Z+(\nabla _{X,Z}^{2}\phi )Y=T(X,Y,Z),  \label{FEN2}
\end{equation}%
clearly $T(X,Y,Z)=T(X,Z,Y)$. For simplicity $T$ will also denote
corresponding $(0,4)$-tensor 
\begin{equation*}
T(X,Y,Z,W)=g(T(X,Y,Z),W).
\end{equation*}%
From the Ricci identity, 
\begin{equation*}
0=R(X,Y,Z,\phi W)-R(X,Y,W,\phi Z)-g((\nabla _{X,Y}^{2}\phi )Z,W)+g((\nabla
_{Y,X}^{2}\phi )Z,W)
\end{equation*}%
eq. (\ref{FEN2}), and the first Bianchi identity, we find 
\begin{gather}
R(X,Y,Z,\phi W)=R(X,Y,W,\phi Z)+g((\nabla _{X,Y}^{2}\phi )Z,W-g((\nabla
_{Y,X}^{2}\phi )Z,W)  \label{FEN3} \\
=R(X,Y,W,\phi Z)-g((\nabla _{X,Z}^{2}\phi )Y,W){}+{}g((\nabla _{Y,Z}^{2}\phi
)X,W)+  \notag \\
T(X,Z,Y,W)-T(Y,Z,X,W),  \notag \\
R(X,Y,Z,\phi W)=R(X,Z,Y,\phi W)-R(Y,Z,X,\phi W)=R(X,Z,Y,\phi W) \\
{}-{}R(Y,Z,W,\phi X)-g((\nabla _{Y,Z}^{2}\phi )(X,W)+g((\nabla
_{Z,Y}^{2}\phi )X,W),  \notag
\end{gather}%
comparing right hand sides of these equations, we obtain%
\begin{gather}
R(X,Z,Y,\phi W)-R(Y,Z,W,\phi X)-R(X,Y,W,\phi Z)+g((\nabla _{Z,Y}^{2}\phi
)X,W)+  \label{FEN5} \\
{}g((\nabla _{X,Z}^{2}\phi )Y,W)+T(Y,Z,X,W)-T(X,Z,Y,W)=2g((\nabla
_{Y,Z}^{2}\phi )X,W),  \notag
\end{gather}%
we note, that 
\begin{gather}
g((\nabla _{Z,Y}^{2}\phi )X,W)+g((\nabla _{X,Z}^{2}\phi )Y,W)=R(X,Z,Y,\phi W)
\\
{}-R(X,Z,W,\phi Y)+T(Z,X,Y,W),  \notag \\
g((\nabla _{Y,Z}^{2}\phi )X,W)=g((\nabla _{Y,W}^{2}\phi )Z,X)-T(Y,W,Z,X),
\end{gather}%
which being taken into account in (\ref{FEN5}), follow 
\begin{gather}
2R(X,Z,Y,\phi W)-R(X,Y,W,\phi Z)-R(Y,Z,W,\phi X)-R(X,Z,W,\phi Y)+
\label{FEN7} \\
T(Y,Z,X,W)+T(Z,X,Y,W)-T(X,Y,Z,W)+  \notag \\
2T(Y,W,Z,X)=2g((\nabla _{Y,W}^{2}\phi )Z,X).  \notag
\end{gather}%
By straightforward computations 
\begin{multline}
T(Y,Z,X,W)+T(Z,X,Y,W)-T(X,Y,Z,W)+2T(Y,W,Z,X)=  \label{pioT} \\
C(X,Y,Z,W)+2g(\phi Y,W)g(HX,Z)+2g(\phi X,Z)g(HY,W)+ \\
2g(\phi X,W)g(HY,Z)+2g(\phi Z,W)g(HX,Y)+ \\
2g(\phi Z,X)g(Y,\phi ^{2}W)+\eta (X)\eta (Y)g(\phi Z,W){}-\eta (Z)\eta
(Y)g(\phi X,W),
\end{multline}%
where 
\begin{gather}
C(X,Y,Z,W)=-\eta (Y)g((\nabla _{Z}\phi )X,W)+\eta (Y)g((\nabla _{X}\phi )Z,W)
\\
{}-2\eta (W)g((\nabla _{Y}\phi )Z,X).  \notag
\end{gather}%
The anti-symmetrization of (\ref{FEN7}), in $Y$ and $W$, and the first
Bianchi identity, follow 
\begin{gather*}
3R(\phi X,Z,Y,W)+3R(X,\phi Z,Y,W)+3R(X,Z,\phi Y,W)+3R(X,Z,Y,\phi W)+ \\
4g(\phi Y,W)g(HX,Z)+4g(\phi X,Z)g(HY,W)+2g(\phi X,W)g(HY,Z) \\
{}-2g(\phi X,Y)g(HW,Z)+2g(\phi Z,W)g(HX,Y)-2g(\phi Z,Y)g(HX,W)=0,
\end{gather*}%
now (\ref{FEN1}) will be proved, if $H=0$. We shall focus on the proof that $%
H=0$.

For $X=\xi $ ($H\xi =\phi \xi =0$) we obtain 
\begin{gather}
R(\xi ,\phi Z,Y,W)+R(\xi ,Z,\phi Y,W)+R(\xi ,Z,Y,\phi W)=0,  \label{Rxi} \\
-R(\xi ,Z,\phi Y,W)-R(\xi ,\phi Z,Y,W)+\eta (Y)R(\xi ,\phi Z,\xi ,W)+ \\
R(\xi ,\phi Z,\phi Y,\phi W)=0,  \notag
\end{gather}%
hence 
\begin{equation}
R(\xi ,Z,Y,\phi W)+R(\xi ,\phi Z,\phi Y,\phi W)+\eta (Y)R(\xi ,\phi Z,\xi
,W)=0,  \label{Rxi2}
\end{equation}%
and 
\begin{gather}
-R(\xi ,Z,\phi Y,W)+R(\xi ,\phi Z,Y,W)=\eta (W)R(\xi ,Z,\xi ,\phi Y) \\
{}-\eta (W)R(\xi ,\phi Z,\xi ,Y)+\eta (Y)R(\xi ,\phi Z,\xi ,W),  \notag
\end{gather}%
by the last equation, we can simplify (\ref{Rxi}), to 
\begin{gather}
3R(\xi ,\phi Z,Y,W)=2\eta (Y)R(\xi ,\phi Z,\xi ,W)+2\eta (W)R(\xi ,Z,\xi
,\phi Y) \\
{}-\eta (W)R(\xi ,\phi Z,\xi ,Y)-\eta (Y)R(\xi ,Z,\xi ,W).  \notag \\
R(\xi ,Z,\phi Y,\phi W)=0.
\end{gather}

%  Notice that%
% \begin{gather}
% (\nabla _{X}H)\xi =-H^{2}X-HX,  \label{i23} 
% \end{gather}
For $\nabla\xi=-\phi^2+H$, 
\begin{gather}  \label{rxi}
R(Y,Z,\xi ,X)=-g(\nabla _{Y}\phi ^{2})X,Z)+g(\nabla _{Z}\phi ^{2})X,Y) \\
{}-{}g((\nabla _{Y}H)X,Z)+g((\nabla _{Z}H)X,Y),  \notag
\end{gather}
taking cycling sum, by Bianchi identity 
\begin{equation*}
g((\nabla _{Z}H)X,Y)+g((\nabla _{X}H)Y,Z)-g((\nabla _{Y}H)X,Z)=0,
\end{equation*}%
hence 
\begin{gather}
R(Y,Z,\xi ,X) =-g((\nabla _{Y}\phi ^{2})X,Z)+g((\nabla _{Z}\phi
^{2})X,Y)-g(\nabla _{X}H)Y,Z) \\
=\eta (Y) g(X,Z) -\eta (Z) g(X,Y) +\eta(Y)g(X,HZ) -\eta(Z)g(X,HY)  \notag \\
-2\eta (X)g(Z,HY)-g((\nabla _{X}H)Y,Z),  \notag \\
0 = R(\xi ,X,\phi Y,\phi Z) = -2\eta (X)g(H\phi Y,\phi
Z)-g((\nabla_{X}H)\phi Y,\phi Z) \\
=2\eta(X)g(HY,Z)-g((\nabla_XH)\phi Y,\phi Z).  \notag
\end{gather}
Let take local unit eigenvector field $Y$, $\eta(Y)=0$, $H^{2}Y=\lambda Y$,
note that $H^2\phi Y=-\lambda \phi Y$, as $\phi H+H\phi=0$, then 
\begin{gather}
0=R(\xi,X,\phi Y,\phi HY)= 2\lambda\eta(X)-g((\nabla_XH)\phi Y,\phi HY)=
2\lambda\eta(X) \\
{}-\dfrac{1}{2}((\nabla_XH^2)\phi Y,\phi Y) = 2\lambda\eta(X)+\dfrac{1}{2}%
d\lambda(X),  \notag
\end{gather}
so $d\lambda=-4\lambda\eta$, as $X$ is arbitrary, in consequence $\lambda =0$
or $d\eta =0$, and $H=0$ .

To proof (\ref{RRR}), let denote the left hand side of (\ref{FEN1}) by $%
\mathcal{R}_{l}$, then 
\begin{gather}
0=\mathcal{R}_{l}(\phi X,Y,Z,W)+\mathcal{R}_{l}(X,\phi Y,Z,W)-\mathcal{R}%
_{l}(X,Y,\phi Z,W) \\
{}-\mathcal{R}_{l}(X,Y,Z,\phi W)=2R(\phi X,\phi Y,Z,W)-2R(X,Y,\phi Z,\phi W),
\notag
\end{gather}%
now (\ref{R5}) is immediate.
\end{proof}

\begin{proposition}
For nearly Kenmotsu manifold we have 
\begin{gather}  \label{bsteg}
(\nabla_{\phi X}\phi)\phi Y +(\nabla_X\phi)Y -2g(\phi X, Y)\xi+\eta(Y)\phi
X=0.
\end{gather}
\end{proposition}

\begin{proof}
By $\phi^2=-Id +\eta\otimes\xi$, 
\begin{gather}
g((\nabla_X\phi)\phi Y,Z)= g((\nabla_X\phi)Y,\phi Z)+\eta(Z)g(X,Y) + \\
\eta(Y)g(X,Z)-2\eta(X)\eta(Y)\eta(Z),  \notag
\end{gather}
taking into account (\ref{4}), we obtain 
\begin{gather}
g((\nabla_{\phi X}\phi)Y,Z)= g((\nabla_X\phi)Y,\phi Z) +2\eta(Y)g(X,Z) \\
{} - \eta(Z)g(X,Y) -\eta(X)\eta(Y)\eta(Z),  \notag
\end{gather}
the last above identities, together follow (\ref{bsteg}).
\end{proof}

\begin{proposition}
For nearly Kenmotsu manifold, we have the following relations%
\begin{equation}
R(\xi ,X,Y,Z)=\eta (Y)g(X,Z)-\eta (Z)g(X,Y),  \label{i24}
\end{equation}%
\begin{equation}
Ric(\phi Y,\phi Z)=Ric(Y,Z)+2n\eta (Y)\eta (Z),  \label{i27}
\end{equation}%
\begin{equation}
Ric(Z,\phi Y)+Ric(\phi Z,Y)=0,  \label{i26}
\end{equation}%
where Ric indicates the Ricci tensor and $Q$ is the Ricci operator, $%
Ric(X,Y)=g(QX,Y)$.
\end{proposition}

\begin{proof}
Eq. (\ref{i24}) is direct consequence of $\nabla \xi =-\phi ^{2}$, cf. (\ref%
{rxi}).

Let $(E_{0}=\xi ,E_{1},\ldots ,E_{n},E_{n+1},\ldots ,E_{2n})$, $\text{dim}%
M=2n+1$, denote orthonormal $\phi $-frame, $\phi E_{i}=E_{i+n}$, $\phi
E_{i+n}=-E_{i}$, $i=1,\ldots ,n$, then by (\ref{R5}), 
\begin{gather}
Ric(X,Y)=\sum_{i=1}^{n}(R(E_{i},X,Y,E_{i})+R(E_{i+n},X,Y,E_{i+n}))+R(\xi
,X,Y,\xi ) \\
=Ric(\phi X,\phi Y)+\eta (X)Ric(\xi ,Y)-R(\xi ,\phi X,\phi Y,\xi )+R(\xi
,X,Y,\xi )  \notag \\
=Ric(\phi X,\phi Y)+\eta (X)R(\xi ,Y)=Ric(\phi X,\phi Y)+Ric(\xi ,\xi )\eta
(X)\eta (Y),
\end{gather}%
now by (\ref{i24}), $Ric(\xi ,\xi )=-2n$, and the last identity is now
direct consequence of (\ref{i27}).
\end{proof}

Once we know that $d\eta=0$ we are able to describe completely local
structure of nearly Kenmotsu manifold.

\begin{theorem}
\label{Integ} Let $(M,\phi ,\xi ,\eta ,g)$ be a nearly Kenmotsu manifold.
Then

\begin{itemize}
\item[a)] The distribution $D=\ker \eta $ is completely integrable, and
maximall integral submanifolds of $D$ are totally umbilical hypersurfaces,

\item[b)] Maximall integral submanifolds naturally inherits nearly K\"ahler
structure,

\item[c)] Nearly Kenmotsu manifold is locally isometric to warped product of
real line and nearly K\"ahler manifold.
\end{itemize}
\end{theorem}

\begin{proof}
\noindent $\Rightarrow a)$ As Reeb form is closed, it is clear that $D=\ker
\eta$ is completely integrable. If $\tilde M$, denote maximal integral
submanifold of $D$, particularly $dim \tilde M=2n$, then restriction $%
\xi|_{\tilde M}$ is normal vector field, and with respect to such choice of
normal, Weingarten map is $A: \tilde X \mapsto \nabla_{\tilde X}\xi= -X$,
hence $\bar M$ is umbilical.

\noindent $\Rightarrow b)$ Let $J$ be $(1,1)$ tensor field on $\tilde M$
defined by $J\tilde X = \phi \tilde X$. This definition is correct, as $D$
is $\phi$-invariant. It is direct that $J$ is almost complex structure. We
verify 
\begin{equation*}
(\tilde\nabla_{\tilde X}J)\tilde X = \tilde \nabla_{\tilde X}J\tilde X -
J\tilde \nabla_{\tilde X}\tilde X = (\nabla_{\tilde X}\phi)(\tilde X) = 0,
\end{equation*}
as $\eta(\tilde X)=0$, and $\tilde M$ is totally umbilical.

\noindent $\Rightarrow c)$ We can choose coordinate neighborhood $U=I\times
U^{\prime }$, where $I=(-\epsilon,\epsilon)$ is non-empty interval, and $%
U^{\prime }\subset \mathbb{R}^{2n}$ is a disk, For any point $p \in U$, $%
p=(t,x^1,\ldots,x^{2n})=(t,p^{\prime })$, coordinate $t$, can be defined in
the way that $\eta=dt$, $\xi=\frac{\partial}{\partial t}$. If we set $\hat
g=dt^2$ on $I$, then $\pi: p \mapsto t$ is Riemannian submersion, with
fibers $\pi^{-1}(t)=t\times U^{\prime }$. We find that O'Neill tensors $A$,
and $T^0$ vanish, hence $(U,g)$ is warped product $(I\times U^{\prime 2}+f^2
h)$, as metric $h$ we may take $h=\iota_0^*g$, $\iota_0:p^{\prime }\mapsto
0\times U^{\prime }$. As mean vector field $N=-2n\xi$, we have $\pi_*(N)=-2n
\ln |f|\pi_*(\xi)$, hence ${d\ln|f|}/{dt}=1$, $f^2=Ce^{2t}$, and $C=1$, by
our choice of $h$.
\end{proof}

\begin{remark}
In \cite{dileo}, Dileo and Pastore proved special almost Kenmotsu manifold
is locally isometric to warped product of real line and almost Kaehler
manifold.
\end{remark}

\section{Some theorems about nearly Kenmotsu manifolds}

In this section we will show that normal and nearly Kenmostu manifolds of
constant sectional curvature are Kenmotsu. Moreover, nearly Kenmotsu
manifold can be never realized as hypersurface of nearly K\"{a}hler manifold.

\begin{theorem}
Normal nearly Kenmotsu manifold is Kenmotsu.

\begin{proof}
We know that $d\eta =0$, hence nearly Kenmotsu manifold is normal iff $N=0$.
But in view of the Proposition \ref{2-form}, in the case $N=0$, we have 
\begin{equation*}
d\Phi =2\eta \wedge \Phi ,
\end{equation*}%
which means that $M$ is almost Kenmotsu. Now we use the fact that normal
almost Kenmotsu manifold is Kenmotsu.
\end{proof}
\end{theorem}

The almost Hermitian manifold $(N,J,G)$ is called nearly K\"{a}hler if $(%
\bar{\nabla}_{X}J)Y+$ $(\bar{\nabla}_{Y}J)X=0$, $\bar{\nabla}$ denotes the
Levi-Civita connection of $G$, (see for more details \cite{gray1}, \cite%
{gray2}). Simplest facts about nearly K\"{a}hler manifolds: If $N$ is
Hermitian or locally flat then is K\"{a}hler, any four-dimensional nearly K%
\"{a}hler manifold is necessarily K\"{a}hler.

Y. Tashiro \cite{tas} proved that Riemannian hypersurface $(M,g)\subset
(N,G) $, inherits almost contact metric structure $(\phi ,\xi ,\eta ,g)$,
where $(\phi ,\xi ,\eta )$ are defined by 
\begin{equation}
JX=\phi X+\eta (X)N,~~~JN=-\xi ,  \label{w2}
\end{equation}%
where $N$ is normal vector field.

\begin{theorem}
There is no nearly Kenmotsu hypersurface, in nearly K\"ahler manifold.
\end{theorem}

\begin{proof}
Let $A=-\bar{\nabla}N$, be Weingarten map, and $\nabla $ denote Levi-Civita
connection on $M$. From $(\bar{\nabla}_{X}J)Y+(\bar{\nabla}_{Y}J)X=0$,
Gauss-Weingarten equations follow 
\begin{gather}
(\nabla _{X}\phi )Y+(\nabla _{Y}\phi )X-\eta (Y)AX-\eta (X)AY+2h(X,Y)\xi =0,
\\
g(Y,\nabla _{X}\xi )+g(X,\nabla _{Y}\xi )=-h(Y,\phi X)-h(X,\phi Y),
\end{gather}%
if $M$ is nearly Kenmotsu, then 
\begin{gather}
-\eta (Y)\phi X-\eta (X)\phi Y=\eta (Y)AX+\eta (X)AY-2h(X,Y)\xi , \\
h(X,Y)=h(\xi ,\xi )\eta (X)\eta (Y).
\end{gather}%
In consequence $g(\nabla _{X}\xi ,Y)+g(\nabla _{Y}\xi ,X)=0$, which
contradicts with Proposition \ref{killing}.
\end{proof}

\begin{theorem}
\label{gray} \cite{gray2} Let $M$ be a nearly Kaehler manifold with dim $%
M\leq 4.$Then $M$ is Kaehlerian.
\end{theorem}

Using Theorem \ref{gray} we can give following corollary.

\begin{corollary}
There is not exist proper nearly Kenmotsu manifolds for dimension $3$ and $5.
$
\end{corollary}

\end{document}